\documentclass[11pt]{amsart}
\usepackage{xspace}
\usepackage{amssymb}
\usepackage{amsthm}
\usepackage{amsrefs}
\usepackage{amsxtra}
\usepackage{mathrsfs}
\usepackage{enumerate}
\usepackage[matrix,arrow,curve]{xy} 

\newcommand{\downstairs}{\bar}
\newcommand{\proj}{\mathrm{r}}
\newcommand{\sect}{\mathrm{s}}
\newcommand{\iotap}{\tau}

\newcommand{\classR}{\mathscr{R}_\Gamma}
\newcommand{\classT}{\mathscr{T}_\Gamma}
\newcommand{\classH}{\mathscr{H}_\Gamma}

\newcommand{\trivgp}{{\mathbf{1}}}
\newcommand{\classRo}{\mathscr{R}_\trivgp}
\newcommand{\classTo}{\mathscr{T}_\trivgp}

\newcommand{\loccit}{\emph{loc.\ cit.}\xspace}
\newcommand{\bX}{X}
\newcommand{\conn}{^\circ}

\newcommand{\sep}{^{\mathrm{sep}}}
\DeclareMathOperator{\Image}{Im}

\newcommand{\Q}{\mathbb{Q}}

\DeclareMathOperator{\Int}{Int}
\DeclareMathOperator{\Aut}{Aut}
\DeclareMathOperator{\Inn}{Inn}
\DeclareMathOperator{\Gal}{Gal}
\DeclareMathOperator{\Lie}{Lie}
\DeclareMathOperator{\chr}{char}

\newcommand{\inv}{^{-1}}

\newcommand{\lsup}[1]{{}^{#1}}

\newcommand{\maaap}[4]{\ensuremath{{#2}\colon{#3}#1{#4}}}
\newcommand{\map}{\maaap\longrightarrow}

\newcommand{\abmap}[2]{\ensuremath{{#1}\longrightarrow{#2}}}

\newtheorem{thm}[equation]{Theorem}  

\newtheorem{lem}[equation]{Lemma} 
\newtheorem{cor}[equation]{Corollary} 
\newtheorem{notation}[equation]{Notation} 
\theoremstyle{definition} 
 
\theoremstyle{remark} 
\newtheorem{rem}[equation]{Remark}

\title{Root data with group actions}

\date{\today}
\author{Jeffrey D.~Adler}
\address{Department of Mathematics and Statistics\\
American University\\
Washington, DC 20016-8050}
\email[Adler]{jadler@american.edu}
\author{Joshua M.~Lansky} 
\email[Lansky]{lansky@american.edu}

\begin{document}

\begin{abstract}
Suppose $k$ is a field, $G$ is a connected reductive algebraic
$k$-group, $T$ is a maximal $k$-torus in $G$,
and $\Gamma$ is a finite group that acts on $(G,T)$.
From the above,
one obtains a root datum $\Psi$
on which $\Gal(k)\times\Gamma$ acts.
Provided that $\Gamma$ preserves a positive system in $\Psi$, not necessarily
invariant under $\Gal(k)$,
we construct an inverse to this process.
That is, given a root datum on which $\Gal(k)\times\Gamma$ acts appropriately,
we show how to construct a pair $(G,T)$, on which $\Gamma$ acts as above.

Although the pair $(G,T)$ and the action of $\Gamma$ are canonical only up to an
equivalence relation,
we construct a particular pair for which $G$ is $k$-quasisplit and
$\Gamma$ fixes a $\Gal(k)$-stable pinning of $G$.
Using these choices,
we can define a notion of taking ``$\Gamma$-fixed points''
at the level of equivalence classes,
and this process is compatible with
a general ``restriction'' process
for root data with $\Gamma$-action.
\end{abstract}

\maketitle

\section{Introduction}
\label{sec:intro}
Let $k$ be a field with separable closure $k\sep$.
Let $\Gamma$ be a finite group.

Suppose $\Psi$ is a (reduced) based root datum on which the absolute Galois
group $\Gal(k)$ acts.
Then it is well known 
(\cite{springer:lag-article}*{Theorem 6.2.7})
that there exists a connected, reductive, $k$-quasisplit $k$-group $G$,
uniquely determined up to $k$-isomorphism,
such that the root datum of $G$ (with respect to a maximal $k$-torus contained
in a Borel $k$-subgroup) is isomorphic to $\Psi$ and carries the same
action of $\Gal(k)$.
We generalize this result in two directions.
\begin{enumerate}[(A)]
\item
Suppose $G$ is a connected reductive $k$-group, and $T$ is an arbitrary maximal
torus.
Then the \emph{non-based} root datum $\Psi(G,T)$ carries an action of $\Gal(k)$.
We show that one can reverse this process.
That is, given a root datum $\Psi$ with an action of $\Gal(k)$, one
can obtain a pair $(G,T)$ that gives rise to $\Psi$.
In general, the pair $(G,T)$ need not be uniquely determined
up to $k$-isomorphism.
However, we can always choose $G$ to be $k$-quasisplit,
and all possibilities for $G$ must be $k$-inner forms of each other.
\item
Now suppose that $\Gamma$ acts 
on a pair $(G,T)$ as above via $k$-automorphisms.
Then $\Gamma$ acts on the root datum $\Psi(G,T)$, and the actions
of $\Gamma$ and $\Gal(k)$ commute.
We show that one can reverse this process
under mild conditions.
That is,
suppose that $\Psi$ is a root datum with an action of $\Gal(k)\times \Gamma$.
Assume that $\Gamma$ (but not necessarily $\Gal(k)$) preserves a base.
Then one can obtain a pair $(G,T)$ as above, carrying an action of $\Gamma$
via $k$-automorphisms.
That is, under appropriate conditions, we can lift an action of $\Gamma$
from a root datum to a reductive group.
Moreover, one can choose $G$ to be $k$-quasisplit, and can choose the action
of $\Gamma$ to preserve a pinning.
\end{enumerate}

The above are all contained in our main result,
Theorem \ref{thm:root-data-to-group}.
In order to state it more precisely,
let us consider the collection of abstract root data $\Psi$ that carry an action of
$\Gal(k) \times \Gamma$ such that
the action of $\Gamma$ stabilizes a base for $\Psi$.
We consider two data $\Psi$ and $\Psi'$ with such actions to be
\emph{equivalent}
if there is a $\Gal(k)\times\Gamma$-equivariant
isomorphism $\abmap{\Psi}{\Psi'}$.
Let $\classR$ denote the set of equivalence classes of reduced data
with such actions.

Let $G$ be a connected reductive $k$-group and $T\subseteq G$ a maximal $k$-torus.
Suppose there exists some Borel subgroup $B\subseteq G$ (not necessarily defined over $k$)
containing $T$,
and a homomorphism $\varphi$ from $\Gamma$ to the group $\Aut_k(G,B,T)$ of $k$-automorphisms of
$G$ stabilizing $T$ and $B$. Suppose $G'$, $T'$, and
$\varphi'$ are defined similarly.
We say that the triples $(G,T,\varphi)$ and $(G',T',\varphi')$
are \emph{equivalent}
if there exists an isomorphism $\map{\nu}{G}{G'}$
whose restriction
gives a $\Gamma$-equivariant $k$-isomorphism $T\longrightarrow T'$.
(In this situation, $\nu$ must be an inner twisting by~\cite{springer:corvallis}*{\S3.2}.)
Let $\classT$ be the set of equivalence classes of such triples $(G,T,\varphi)$.

A triple $(G,T,\varphi)$ as above naturally determines a root datum with appropriate actions of
$\Gal(k)$ and $\Gamma$, hence an element of $\classR$.
It is easily seen that if
$(G',T',\varphi')$ and $(G,T,\varphi)$ are equivalent, then
they determine the same class in $\classR$.
Hence we have a natural
map $\map{\proj_\Gamma}{\classT}{\classR}$.

Our main result is the following:
\begin{thm}
\label{thm:root-data-to-group}
The map $\map{\proj_\Gamma}{\classT}{\classR}$ is a bijection.
\end{thm}
In \S\ref{sec:cohomology}, we give cohomological
descriptions of $\classR$ and the inverse of the map $\proj_\Gamma$,
and thus obtain a cohomological description of $\classT$.

In~\cite{adler-lansky:lifting}, we introduce the notion of a \emph{parascopic group}
$\downstairs G$
for a pair $(G,\Gamma)$, where $\Gamma$ is a finite group acting via $k$-automorphisms
(in a certain specified way) on a maximal $k$-torus of the reductive $k$-group $G$.
This notion axiomatizes and generalizes the relationship between $G$ and the connected part
$(G^\Gamma)\conn$ of the group of $\Gamma$-fixed points in the case
where the action of $\Gamma$ extends to one on $G$. When $\downstairs G$
is parascopic for $(G,\Gamma)$, we construct a canonical lifting of semisimple
stable conjugacy classes from $\downstairs G^\wedge$ to $G^\wedge$, $k$-groups that
are in duality with $\downstairs G$ and $G$, respectively.
Theorem~\ref{thm:root-data-to-group} is a key tool in our further study of
this lifting in~\cite{adler-lansky:lifting2}. In particular,
in many settings, this result allows one to factor such a lifting into a composition of simpler ones
which are much more readily understood.

We also prove a generalization of a second well-known result.
Suppose $\Gamma$ is a finite
group acting algebraically on a connected reductive group $G$,
fixing a Borel-torus pair $(B,T)$ in $G$
and a pinning for $(G,B,T)$.
Then the root system of the connected part
$\downstairs G:= (G^\Gamma)\conn$
of the group of fixed points is obtained as follows.
The set of restrictions of roots of $G$ from $T$ to
$\downstairs T := (T^\Gamma)\conn$
is a root system, not necessarily reduced,
but there is a preferred way to choose a maximal reduced subsystem.
The above is well known,
but in 
Theorem~\ref{thm:restriction} and Lemma~\ref{lem:weyl-fixed-pinning}
we describe the root datum, not just the root system,
of $\downstairs G$ with respect to $\downstairs T$.

Theorem \ref{thm:compatibility}
says that the process of passing from $(G,T)$
to $(\downstairs G,\downstairs T)$
is compatible with the bijection of Theorem \ref{thm:root-data-to-group}.
To state this result more precisely,
suppose that the triple $(G,T,\varphi)$ represents an element of $\classT$.
Then we know \cite{adler-lansky:lifting}*{Proposition 3.5}
that
$\downstairs{G}$
is a connected reductive $k$-group,
and
$\downstairs{T}$ 
is a maximal $k$-torus in $\downstairs{G}$.
Thus, if we let ``$1$'' represent the map from the trivial
group $\trivgp$ to $\Aut(G)$,
then the triple $(\downstairs{G},\downstairs{T},1)$
represents an element of $\classTo$.
The equivalence class of $(G,T,\varphi)$ does not determine
that of
$(\downstairs{G},\downstairs{T},1)$,
or even the $k\sep$-isomorphism class of $\downstairs G$.
Nonetheless, we can obtain
a well-defined map $\abmap\classT\classTo$ as follows:
From Remark \ref{rem:fixed-pinning},
we will see that
every class in $\classT$
contains a triple $(G,T,\varphi)$ such $G$ is $k$-quasisplit
and $\varphi$ fixes a $\Gal(k)$-invariant pinning.
Use this choice of triple to define
$\downstairs G$ and $\downstairs T$,
and it is straightforward to show that our choices determine
$\downstairs G$ and $\downstairs T$ up to $k$-isomorphism.

Suppose that the root datum $\Psi$ represents an element of $\classR$.
We will see in \S\ref{sec:root-data} that the action of $\Gamma$
on $\Psi$ allows us to construct a ``restricted'' root datum
$\downstairs{\Psi}$ that has a preferred
choice
of reduced subdatum $\downstairs{\Psi}'$.
We thus obtain a map $\abmap\classR\classRo$.

\begin{thm}
\label{thm:compatibility}
Our maps
$\abmap{\classT}{\classTo}$
and
$\abmap{\classR}{\classRo}$
above
are compatible with the maps
of Theorem \ref{thm:root-data-to-group},
in the sense that the following
diagram commutes:
\begin{equation*}
\begin{xy}
\xymatrix{
\classT \ar[d] \ar[r]^{\proj_\Gamma} & \classR \ar[d]  \\
\classTo  \ar[r]^{\proj_\trivgp} & \classRo 
}
\end{xy}
\end{equation*}
\end{thm}
We prove both theorems in \S\ref{sec:proofs}.

Some of the results in this paper have recently appeared elsewhere,
and others are generalizations of existing results.
For example, as indicated above, \cite{springer:lag-article}
presents Statement (A)
in the case where $G$ is quasi-split over $k$ and $T$ is contained
in a Borel $k$-subgroup of $G$.
Corollary \ref{cor:weyl-isomorphism}
has appeared as \cite{haines:satake}*{Lemma 4.2(2)}.
Lemma \ref{lem:weyl-fixed-pinning} overlaps to some extent with
\cite{haines:satake}*{Proposition 4.1}, but
our proof is quite different.
Theorem \ref{thm:restriction}
is similar to \cite{haines:satake}*{Lemma 4.2(1)} and
\cite{haines:dualities}*{Theorem A}, except that we describe
the full restricted root datum, and our proof does
not rely on Steinberg \cite{steinberg:endomorphisms}.

\section{Restrictions of root data}
\label{sec:root-data}

Let $\Psi = (X^*,\Phi,X_*,\Phi^\vee)$
be a root datum.
(We do not assume that $\Phi$ is reduced.)
Let $\Gamma$ denote a finite group of automorphisms of $\Psi$.
We assume that there exists a $\Gamma$-stable set $\Delta$
of simple roots in $\Phi$.
Let $V^* = X^* \otimes \Q$ and $V_* = X_* \otimes \Q$.
Let $i^*$ denote the quotient map from $V^*$ to its space
$\downstairs V^* := V^*_\Gamma$ of $\Gamma$-coinvariants.
From \cite{adler-lansky:lifting}*{\S2},
there is an embedding $\map\iota{\downstairs V^*}{V^*}$ with image $V^{*\,\Gamma}$
given by
\[
\iota (\downstairs v) = \frac{1}{|\Gamma|} \sum_{\gamma\in\Gamma} \gamma v,
\]
where $v$ is any preimage in $V^*$ of $\downstairs v \in \downstairs V^*$.
Let $\downstairs X^*$ and $\downstairs \Phi$ denote the images
of $X^*$ and $\Phi$ under $i^*$.
Then $\downstairs X^*$ is $X^*_\Gamma$ modulo torsion,
where $X^*_\Gamma$ denotes the module
of $\Gamma$-coinvariants of $X^*$.
It is straightforward to see that 
$\downstairs X^*$ and $\downstairs X_*:= X_*^\Gamma$,
and thus
$\downstairs V^*$ and $\downstairs V_*:=V_*^\Gamma$,
are in duality via the pairing given by
$\langle \downstairs x, \downstairs\lambda\rangle :=
\langle \iota \downstairs x, i_*\downstairs\lambda\rangle$,
where $\map{i_*}{\downstairs X_*}{X_*}$
is the inclusion map.
With respect to these pairings, 
$i^*$ is the transpose of $i_*$.

For each $\beta\in\Phi$, let $w_\beta$ denote the automorphism of $X^*$ defined by
\[
w_\beta(x) = x-\langle x,\beta^\vee\rangle \beta.
\]
Let $W$ denote the Weyl group of $\Psi$,
i.e., the (finite) subgroup of $\Aut(X^*)$
generated by the $w_\beta$.
Then $\Gamma$ acts naturally on $W$, and $W$ acts on $X^*$.
The group $W^\Gamma$ of $\Gamma$-fixed elements of $W$ acts on
on both $\downstairs V^*$ and $\downstairs X^*$ via the rule
$w(i^* x) := i^*(w(x))$
for $w\in W^\Gamma$ and $x\in X^*$.
Equivalently, for $\downstairs x \in \downstairs X^*$ and $w\in W^\Gamma$,
we have $\iota (w(\downstairs x)) = w(\iota \downstairs x)$.

\begin{lem}[cf.~\cite{steinberg:endomorphisms}*{\S1.32(a)}]
\label{lem:fidelite}
The natural action of $W^\Gamma$ on $\downstairs X^*$ is faithful.
\end{lem}
\begin{proof}
Let $w$ be a nontrivial element of $W^\Gamma$.
Then there exists a positive root $\beta\in \Phi$
such that $w(\beta)$ is negative.
Since $\Gamma$ stabilizes $\Delta$, it follows that
$w(\gamma\cdot\beta) = \gamma\cdot (w \beta)$ is also negative for every $\gamma\in\Gamma$.
Thus $\iota (w(i^*\beta)) $ is a linear combination of roots in $\Delta$
in which all of the coefficients are nonpositive, so $w(i^*\beta)\neq i^*\beta$.
\end{proof}

\begin{notation}
\label{notation:Xi}
For each root $\beta \in \Phi$,
define a $\Gamma$-orbit $\Xi_\beta$ in $\Phi$
as in \cite{adler-lansky:lifting}*{\S5}.
That is,
let $\Xi_\beta = \Gamma \cdot \beta$
if this is an orthogonal set.
Otherwise,
for each $\theta\in\Gamma\cdot\beta$,
there exists a unique root
$\theta'\neq\theta$ in $\Gamma\cdot\beta$ such that $\theta$ and $\theta'$ are
not orthogonal.
Moreover, $\theta+\theta'$ is a root in $\Phi$ and does
not belong to $\Gamma\cdot\beta$.
In this case, let $\Xi_\beta = \{\theta+\theta'\mid\theta\in\Gamma\cdot\beta\}$.
\end{notation}

\begin{rem}
\label{rem:weyl-lift}
Thus, in all cases, $\Xi_\beta$ is an orthogonal $\Gamma$-orbit of roots.
\end{rem}

\begin{lem}
\label{lem:orbit}
If $\alpha\in\downstairs\Phi$,
then ${i^*}\inv (\alpha)$ is a $\Gamma$-orbit of roots in $\Phi$.
\end{lem}
\begin{proof}
This argument is similar to but more general than that given in the proof of~\cite{springer:lag}*{Lemma 10.3.2(ii)}.
Suppose $\beta\in\Phi$ and $i^*(\beta) = \alpha$.
Then clearly $i^*\theta = \alpha$ for any
$\theta\in\Gamma\cdot\beta$.

Now suppose $\beta'\in \Phi$, $\beta'\neq \beta$, and $i^*\beta'=\alpha$.
Since $\iota(i^*(\beta'-\beta)) = 0$ and since $\Gamma$ preserves $\Delta$,
when $\beta'-\beta$ is written as a linear combination of simple roots,
the coefficients must
sum to $0$.
In particular, $\beta'-\beta\notin\Phi$. Since $\beta'$ cannot be a multiple of $\beta$, we have that 
$\langle \beta',\beta^\vee\rangle\leq 0$
by standard results about root systems.
Similarly, $\langle \beta',\theta^\vee\rangle\leq 0$ for all $\theta\neq\beta'$ in $\Gamma\cdot\beta$.
Therefore,
\[
\Bigl\langle \beta' - \beta, \sum_{\theta\in \Gamma\cdot \beta} \theta^\vee \Bigr\rangle
=
\Bigl\langle \beta' - \beta, i_*\bigl(\sum_{\theta\in \Gamma\cdot \beta} \theta^\vee\bigr) \Bigr\rangle
=
\Bigl\langle i^*(\beta' - \beta), \sum_{\theta\in \Gamma\cdot \beta} \theta^\vee \Bigr\rangle ,
\]
and since $i^*(\beta' - \beta) = 0$, this pairing vanishes.
Thus
$
\sum_{\theta\in \Gamma\cdot \beta} \langle \beta',  \theta^\vee \rangle =
\sum_{\theta\in \Gamma\cdot \beta} \langle \beta, \theta^\vee \rangle = 2
$
or $1$, 
depending on whether or not $\Gamma\cdot\beta$ is orthogonal.
(This follows from the properties of root orbits discussed in~\cite{adler-lansky:lifting}*{\S5}.)
Since $\langle \beta',\theta^\vee\rangle\leq 0$ for all $\theta\neq\beta'$ in $\Gamma\cdot\beta$, it follows that
$\beta' \in\Gamma\cdot\beta$.
\end{proof}

For each $\alpha\in\downstairs\Phi$, define
\[
\alpha^\vee = \frac{|\Gamma\cdot\beta|}{|\Xi_\beta|}\sum_{\xi\in\Xi_\beta}\xi^\vee\in \downstairs X_*,
\]
where $\beta$ is any element of $\Phi$ such that $i^*\beta = \alpha$.
The element of
$\downstairs X_*$ defined by the above formula is independent of the particular choice of $\beta$ by Lemma~\ref{lem:orbit}.
Note that $\alpha^\vee$ does indeed lie in $\downstairs X_*$
since $|\Gamma\cdot\beta|/|\Xi_\beta| = 1$ or $2$.
Let $\downstairs\Phi^\vee = \{\alpha^\vee\mid \alpha\in\downstairs\Phi\}$.

\begin{thm}
\label{thm:restriction}
With the above notation,
$\downstairs\Psi
:= (\downstairs X^*,\downstairs \Phi,\downstairs X_*,\downstairs \Phi^\vee)$
is a root datum.
\end{thm}

\begin{rem}
If $\Psi$ comes equipped with an action of $\Gal(k)$, and the action
of $\Gamma$ commutes with that of $\Gal(k)$, then it is clear
that the action of $\Gal(k)$ preserves $\downstairs\Psi$.
\end{rem}

\begin{proof}[Proof of Theorem \ref{thm:restriction}]
According to~\cite{springer:corvallis}*{\S1.1}, it suffices to show that
\begin{itemize}
\item
$\downstairs X^*$ and $\downstairs X_*$ are in duality
(which we have already observed),
\item
$\langle \alpha,\alpha^\vee\rangle = 2$ for all $\alpha\in\downstairs\Phi$,
and
\item
The automorphisms $w_\alpha$ of $\downstairs X^*$ of the form
\[
w_\alpha(\downstairs x)
= \downstairs x - \langle \downstairs x,\alpha^\vee\rangle\alpha\qquad
	\text{(for $\alpha\in\downstairs\Phi$)}
\]
stabilize $\downstairs\Phi$
and generate a finite subgroup of $\Aut(\downstairs X^*)$.
\end{itemize}

Let $\alpha\in\downstairs\Phi$. Choose $\beta\in\Phi$ such that $i^*\beta = \alpha$, and choose $\xi_0\in\Xi_\beta$.
Then we have
\begin{align*}
\langle \alpha,\alpha^\vee\rangle
&=
\langle \iota\alpha,i_*\alpha^\vee\rangle\\
&=
\Bigl\langle
\frac{1}{|\Gamma\cdot\beta|}\sum_{\theta\in\Gamma\cdot\beta}\theta,
\frac{|\Gamma\cdot\beta|}{|\Xi_\beta|}\sum_{\xi\in\Xi_\beta}\xi^\vee
\Bigr\rangle
\\
&=
\frac{1}{|\Xi_\beta|}
\Bigl\langle
\sum_{\theta\in\Gamma\cdot\beta}\theta,
\sum_{\xi\in\Xi_\beta}\xi^\vee
\Bigr\rangle
\\
&=
\frac{1}{|\Xi_\beta|}
\Bigl\langle
\sum_{\xi'\in\Xi_\beta}\xi',
\sum_{\theta\in\Xi_\beta}\xi^\vee
\Bigr\rangle
&& \text{(by the definition of $\Xi_\beta$)}
\\
&=
\langle \xi_0 ,\xi_0^\vee\rangle
&& \text{(by Remark \ref{rem:weyl-lift})}
\\
&= 2,
\end{align*}
as desired.

Now let $\downstairs x\in \downstairs X^*$, and 
choose $x\in X^*$ such that $i^* x = \downstairs x$. Then
\begin{equation}
\label{eq:integral}
\begin{aligned}
\langle \downstairs x,\alpha^\vee\rangle &= 
\langle x,i_* \alpha^\vee\rangle  \\
&= \Bigr\langle x,
\frac{|\Gamma\cdot\beta|}{|\Xi_\beta|}\sum_{\xi\in\Xi_\beta}{\xi}^\vee\Bigr\rangle
&& \text{(by Remark \ref{rem:weyl-lift})}\\
&= \frac{|\Gamma\cdot\beta|}{|\Xi_\beta|}\sum_{\xi\in\Xi_\beta}
\langle x, \xi^\vee\rangle.
\end{aligned}
\end{equation}
It follows that
\begin{align*}
w_{\alpha}(\downstairs x)
&= \downstairs x- \langle \downstairs x,\alpha^\vee \rangle \alpha\\
&= i^* x - \langle \downstairs x,\alpha^\vee \rangle i^*\beta\\
&= i^* x -  \frac{\langle  \downstairs x,\alpha^\vee \rangle}{|\Gamma\cdot\beta|}
i^* \Bigl(\sum_{\theta\in\Gamma\cdot\beta}\theta\Bigr) 
&& \text{(by Lemma \ref{lem:orbit})} \\
&= i^* x  - \frac{\langle \downstairs x, \alpha^\vee \rangle}{|\Gamma\cdot\beta|}
i^*\Bigl(\sum_{\xi'\in\Xi_\beta}\xi'\Bigr)
&& \text{(by the definition of $\Xi_\beta$)} \\
&=
i^* x  - \frac{1}{|\Xi_\beta|}
\sum_{\xi\in\Xi_\beta}
\langle x , \xi^\vee\rangle
i^*
\Bigl(
\sum_{\xi'\in\Xi_\beta}
	\xi'
\Bigr)
&& \text{(by \eqref{eq:integral})} \\
&= i^* x  - \sum_{\xi\in\Xi_\beta}\langle x , \xi^\vee\rangle
i^*
\Bigl(
	\frac{1}{|\Xi_\beta|}
	\sum_{\xi'\in\Xi_\beta} \xi'
\Bigr).
\end{align*}
But by Remark~\ref{rem:weyl-lift}, for any $\xi\in\Xi_\beta$,
\[
i^*\Bigl( \frac{1}{|\Xi_\beta|} \sum_{\xi'\in\Xi_\beta} \xi'\Bigr) = i^*\xi,
\]
so we have
\begin{equation}
\label{eq:weyl-image}
w_{\alpha}(\downstairs x) = i^*\Bigl(x  - \sum_{\xi\in\Xi_\beta}
\left\langle x , \xi^\vee\right\rangle \xi\Bigr).
\end{equation}
Also by Remark~\ref{rem:weyl-lift}, the reflections $w_{\xi}\in W$ for $\xi\in\Xi_\beta$
all commute with one another.
If $w$ denotes their product, then
\[
x - \sum_{\xi\in\Xi_\beta}\langle x,\xi^\vee\rangle\xi  = w (x),
\]
so by \eqref{eq:weyl-image}, we have
\begin{equation}
\label{eq:embedding}
w_{\alpha}(\downstairs x)
= i^*(w(x)).
\end{equation}
In particular,
if $\alpha'\in\downstairs \Phi$,
and $\beta'\in \Phi$ satisfies $i^*\beta' = \alpha'$, then
\[
w_{\alpha}(\alpha')
= i^*(w(\beta'))\in i^*(\Phi) = \downstairs \Phi ,
\]
so $\downstairs \Phi$ is stable under the action of $w_\alpha$, as desired.

It remains to show that the group
$\downstairs W
:= \langle w_\alpha\mid\alpha\in\downstairs\Phi\rangle\subset \Aut(\downstairs X^*)$
is finite.
To accomplish this,
we show that $\downstairs W$ embeds naturally in the finite group
$W^\Gamma$.
By Lemma~\ref{lem:fidelite}, there is a natural injection
\[
W^\Gamma\longrightarrow\Aut(\downstairs X^*).
\]
To construct an embedding $\downstairs W\longrightarrow W^\Gamma$,
it is therefore
enough to show that the image of this injection contains $\downstairs W$.
Thus, given $\downstairs w\in \downstairs W$,
we will show that there exists $w \in W^\Gamma$
whose action on $\downstairs X^*$ coincides with that of $\downstairs w$.
It suffices to prove the existence of $w$ only in the case in which
$\downstairs w$ is a reflection
$w_\alpha$ through a root $\alpha\in\downstairs\Phi$.
In this case, let
$w = \prod_{\xi\in\Xi_\beta}w_\xi$,
where 
$\beta\in\Phi$ is such that $i^*\beta = \alpha$.
It follows from Remark~\ref{rem:weyl-lift} that $w\in W^\Gamma$,
and it follows from \eqref{eq:embedding} that for any $x\in X^*$,
\[
w_\alpha(i^* x) = i^*(w(x)) = w(i^* x).
\]
This establishes the existence of the desired embedding.
\end{proof}

\begin{rem}
\label{rem:reduced}
If $\Phi$ is reduced, then so is
the root system $\downstairs\Phi$ constructed above,
unless
$\Phi$ has an irreducible factor of type $A_{2n}$ whose stabilizer in $\Gamma$
acts upon it nontrivially.
To see this, it is easy
to reduce to the case where $\Phi$ is irreducible
and $\Gamma$ is cyclic (see~\cite{adler-lansky:lifting}*{Proposition 3.5}).
The result then follows from 
\cite{kottwitz-shelstad:twisted-endoscopy}*{\S1.3}.
\end{rem}

\begin{rem}
\label{rem:reduced-subsystem}
There is a way to choose a maximal reduced subsystem
that we will later see is preferred. (See Lemma~\ref{lem:weyl-fixed-pinning}.)
Specficially, take only the nondivisible (resp.\ nonmultipliable)
roots of $\downstairs\Phi$ according as $\chr k$ is not two
(resp.\ two).
\end{rem}

\begin{lem}
\label{lem:pos-systems}
The map $i^*$ induces a bijection between
the set of $\Gamma$-invariant positive systems in $\Phi$
and the set of positive systems in $\downstairs\Phi$.
\end{lem}

\begin{proof}
Let $\Pi\subseteq\Phi$ be a $\Gamma$-invariant positive system.
Let $\downstairs\Pi = i^*(\Pi)\subseteq \downstairs\Phi$.
Then there is some vector $v\in V_*$ such that for every root $\beta\in\Phi$,
we have that
$\langle \beta, v \rangle \neq 0$,
and
$\langle \beta,v \rangle > 0$ if and only if $\beta\in\Pi$.
Since $\Pi$ is $\Gamma$-invariant, we may replace $v$ by
$\sum_{\gamma\in\Gamma} \gamma v$,
and thus assume that $v$ is $\Gamma$-invariant,
and so lies in $\downstairs V_*$.
Suppose $\alpha\in\downstairs \Phi$.
Then $\alpha = i^*\beta$ for some $\beta\in\Phi$,
so $\langle \alpha, v \rangle = \langle \beta, v\rangle$.
Thus, $\langle \alpha, v \rangle \neq 0$,
and
$\langle \alpha, v \rangle > 0$
if and only if $\alpha\in \downstairs\Pi$.
This shows that $\downstairs\Pi$ is a positive system in $\downstairs\Phi$.

Conversely, suppose that $\downstairs\Pi\subseteq\downstairs\Phi$
is a positive system,
and let $\Pi = {i^*}\inv \downstairs \Pi$.
Then there is some vector $\downstairs v \in \downstairs V_*$ such that for every
root $\alpha \in \downstairs \Phi$,
we have that
$\langle \alpha, \downstairs v\rangle \neq 0$, and
$\langle \alpha, \downstairs v\rangle > 0$
if and only if $\alpha \in \downstairs \Pi$.
For every root $\beta\in \Phi$, we have
$\langle \beta, i_* v \rangle = \langle i^*\beta, v \rangle$,
which is never zero, and is positive if and only if $\beta \in \Pi$.
Thus, $\Pi\subset\Phi$ is a positive system.
Since $i_*v$ is $\Gamma$-invariant, so is $\Pi$.
\end{proof}

\begin{cor}
\label{cor:weyl-isomorphism}
Let $\downstairs W$ be the Weyl group  of $\downstairs\Psi$.
Then the embedding of $\downstairs W$ into $W^\Gamma$
in the proof of Theorem~\ref{thm:restriction} is an isomorphism.
\end{cor}

\begin{proof}
Since
$\downstairs W$ acts simply transitively on the set of
positive systems in $\downstairs \Phi$,
and
$W^\Gamma$ acts simply transitively on the set of
$\Gamma$-invariant
positive systems in $\Phi$,
the result follows from Lemma \ref{lem:pos-systems}.
\end{proof}

\section{Automorphisms of root data and automorphisms of reductive groups}
\label{sec:automorphisms}
Let $\Psi = (X^*,\Phi,X_*,\Phi^\vee)$ be a root datum on which a group $\Lambda$ acts via automorphisms.
Choosing a root basis $\Delta$ of $\Phi$, we obtain a corresponding based root datum $\dot\Psi$.
Then $\dot\Psi$ also carries an action of $\Lambda$. Namely, for $\sigma\in\Lambda$, there exists
a unique element
$c(\sigma)$ in the Weyl group $W(\Psi)$ of $\Psi$ such that $\sigma(\Delta) = c(\sigma)(\Delta)$.
If we define $\sigma^\star$ to be the automorphism of $\dot\Psi$ given by
\begin{equation}
\label{eq:*-action}
\sigma^\star\chi = c(\sigma)\inv(\sigma\chi)
\end{equation}
for $\chi\in X^*$,
then the action of $\Lambda$ on $\dot\Psi$ is given by
$\sigma\mapsto\sigma^\star$.

Since $\Lambda$ acts on $\Psi$ and $\dot\Psi$,
it acts naturally on $\Aut(\Psi)$ and $\Aut(\dot\Psi)$, as well as on
the Weyl groups $W(\Psi)\subset\Aut(\Psi)$ and $W(\dot\Psi)\subset\Aut(\dot\Psi)$. 
Just as the actions of $\Lambda$ on $\dot\Psi$ and on $\Psi$ differ,
so the actions of $\Lambda$ on $W(\dot\Psi)$ and on $W(\Psi)$ differ,
even though the Weyl groups themselves are equal.
For $\sigma\in \Lambda$ and $w\in W(\dot\Psi)$, 
let $(\sigma ,w)\mapsto\sigma^\star(w)$ denote the action
of $\Lambda$ on $W(\dot\Psi)$.
Then we have
\[
\sigma w = c(\sigma)\, (\sigma^\star w) \, c(\sigma)\inv .
\]
One can check readily that map $\map{c}{\Lambda}{W(\dot\Psi)}$
is a cocycle in $Z^1(\Lambda,W(\dot\Psi))$.

We now turn our attention to based root data arising from reductive algebraic groups.
Let $G$ be a connected reductive $k$-group,
$B$ a Borel subgroup of $G$, and $T\subseteq B$ a maximal torus of $G$.
Let $\dot\Psi(G,B,T)$ denote the corresponding based root datum.

Any map $\vartheta \in\Aut(G)$ determines an obvious isomorphism
\[
\map{\vartheta ^*}{\dot\Psi(G,B,T)}{\dot\Psi(G,\vartheta (B),\vartheta (T))}.
\]
There is a natural homomorphism
\[
\map{\pi}{\Aut(G)}{\Aut(\dot\Psi(G,B,T))}
\]
defined as follows.
For $\vartheta \in\Aut(G)$, choose $g_\vartheta \in G(k\sep)$ such that $\Int(g_\vartheta )$ takes
$\vartheta (B)$ to $B$, and $\vartheta (T)$ to $T$. Then $\Int(g_\vartheta )\circ\vartheta $ stabilizes $B$ and $T$,
and we let $\pi(\vartheta )$ be the automorphism $(\Int(g_\vartheta )\circ\vartheta )^*$ of $\dot\Psi(G,B,T)$
(which is, in fact, independent of the choice of $g_\vartheta $).

Now suppose that $T$ is defined over $k$. Then
an element $\sigma\in \Gal(k)$ naturally determines an automorphism of $\Psi(G,T)$ hence an automorphism
$\sigma^\star$ of $\dot\Psi(G,B,T)$ as defined in \eqref{eq:*-action}.
We thus obtain an action of $\Gal(k)$ on $\dot\Psi(G,B,T)$, hence one on $\Aut(\dot\Psi(G,B,T))$ as above.
These actions are independent of the particular choice of $B$ and $T$
in the sense that if $g\in G(k\sep)$ and $\sigma\in\Gal(k)$, then we have
\begin{equation}
\label{eq:*-action-equivariance}
\sigma^\star \circ \Int (g)^* = \Int(g)^*\circ \sigma^\star ,
\end{equation}
where we use the notation $\sigma^\star$ to denote both the action of $\sigma$ on $\dot\Psi(G,B,T)$
and on $\dot\Psi(G,\lsup{g}B,\lsup{g}T)$.

There is a well-known exact sequence
\begin{equation}
\label{eq:automorphisms}
1 \longrightarrow  \Inn(G) \longrightarrow  \Aut(G) \stackrel{\pi}{\longrightarrow}
\Aut(\dot\Psi(G,B,T)) \longrightarrow  1 .
\end{equation}
We note that the homomorphisms in \eqref{eq:automorphisms} are $\Gal(k)$-equivariant.

\begin{rem}
\label{rem:splitting}
Let $\Delta$ be the set of simple roots for $(G,B,T)$. Let $\{ X_\alpha\}_{\alpha\in\Delta}\subset \Lie(G)(k\sep)$
be a pinning. It is well known~\cite{springer:corvallis}*{Cor.~2.14}
that $\{ X_\alpha\}$ determines a unique splitting $\psi$ of \eqref{eq:automorphisms}.
Namely, if $f\in\Aut(\dot\Psi(G,B,T))$, define $\psi(f)$ to be the automorphism of $G$ such that 
\begin{itemize}
\item $\psi(f)$ stabilizes $B$ and $T$,
\item the restriction of $\psi(f)$ to $T$ is determined by the automorphism of $\bX^*(T)$ given by $f$, and
\item $\psi(f) (X_\alpha) = X_{f(\alpha)}$.
\end{itemize}
Thus $\Image(\psi)$ lies in the subgroup $\Aut(G,B,T,\{ X_\alpha\})$ of $\Aut(G)$ consisting of automorphisms
that stabilize $B$, $T$, and the pinning $\{ X_\alpha\}$.
If $B$ and $T$ are defined over $k$, and $\{ X_\alpha\}$ is $\Gal(k)$-stable,
it follows from~\cite{demazure:sga3-24}*{\S3.10}
that $\psi$ is $\Gal(k)$-equivariant.
\end{rem}

\section{Proofs of Theorems}
\label{sec:proofs}

\begin{proof}[Proof of Theorem \ref{thm:root-data-to-group}]
Consider an abstract root datum $\Psi=(X^*,\Phi,X_*,\Phi^\vee)$ with an action of $\Gal(k)\times\Gamma$. Suppose
that $\Delta$ is a $\Gamma$-stable base for $\Psi$. Let $\dot\Psi$ be the corresponding based root datum.
As discussed in~\S\ref{sec:automorphisms}, the action of $\Gal(k)\times\Gamma$ on $\Psi$ determines one of
$\Gal(k)\times\Gamma$ on $\dot\Psi$.
Since $\Delta$ is $\Gamma$-stable, the actions of $\Gamma$ on $\Psi$ and $\dot\Psi$ coincide.
In the notation of \eqref{eq:*-action} with $\Lambda = \Gal(k)$, the elements $c(\sigma)\in W(\dot\Psi)$ that arise from the action of $\Gal(k)$ on $\dot\Psi$
must lie in $W(\dot\Psi)^\Gamma$ since this action commutes with that of $\Gamma$.
Therefore, the map $\map{c}{\Gal(k)}{W(\dot\Psi)^\Gamma}$
is a cocycle in $Z^1(k,W(\dot\Psi)^\Gamma)$.
We note that the $\Gal(k)\times\Gamma$-isomorphism class of $\dot\Psi$ depends only on that of $\Psi$.

By~\cite{springer:lag-article}*{Theorem 6.2.7}, there exists
a triple $(G,B_0,T_0)$, unique up to $k$-isomorphism,
consisting of a $k$-quasisplit connected reductive group $G$, a Borel $k$-subgroup $B_0$ of $G$,
and a maximal $k$-torus $T_0$ of $B_0$,
such that the associated based root datum $\dot\Psi(G,B_0,T_0)$ is $\Gal(k)$-isomorphic to $\dot\Psi$.
We will identify $\dot\Psi$ and $\dot\Psi(G,B_0,T_0)$ via such an isomorphism.

Let $\{X_\alpha\}$ be a $\Gal(k)$-stable pinning for $G$ relative to $B_0$ and $T_0$. 
The action of $\Gamma$ on $\dot\Psi$ determines a homomorphism $\map{\phi}{\Gamma}{\Aut(\dot\Psi)}$.
Let $\varphi$ be the composition
\[
\varphi: \Gamma \stackrel{\phi}{\longrightarrow} \Aut(\dot\Psi) = \Aut(\dot\Psi(G,B_0,T_0))
\stackrel{\psi}{\longrightarrow}\Aut(G,B_0,T_0,\{ X_\alpha\}) ,
\]
where $\map{\psi}{\Aut(\dot\Psi(G,B_0,T_0))}{\Aut(G,B_0,T_0,\{ X_\alpha\})}$ is the homomorphism
from Remark~\ref{rem:splitting}.

Let $\downstairs G = \bigl(G^{\varphi(\Gamma)}\bigr)\conn$
and
$\downstairs T_0 = \bigl(T_0^{\varphi(\Gamma)}\bigr)\conn$.
By Lemma \ref{lem:weyl-fixed-pinning},
$\downstairs G$ is a $k$-quasisplit reductive group,
$\downstairs T_0$ a maximal $k$-torus of $\downstairs G$,
and
\[
W(\dot\Psi)^\Gamma = W(G,T_0)^{\varphi(\Gamma)} = W(\downstairs G,\downstairs T_0).
\]
Thus we may view $c$ as a cocycle in $Z^1(k,W(\downstairs G,\downstairs T_0))$. 

By~\cite{raghunathan:tori}*{Theorem 1.1}, there is some $g\in\downstairs G(k\sep)$ such that for all
$\sigma\in\Gal(k)$, $g^{-1}\sigma(g)$ lies in the normalizer $N_{\downstairs G}(\downstairs T_0)(k\sep)$, and the
image of $g^{-1}\sigma(g)$ in $W(\downstairs G,\downstairs T_0)$ is equal to $c(\sigma)$.
Let $T = \lsup{g}T_0$
and $B = \lsup{g}B_0$.
Since $g$ is $\varphi(\Gamma)$-fixed, $T$ is a $\varphi(\Gamma)$-stable
maximal $k$-torus of $G$, and $B$ is a $\varphi(\Gamma)$-stable
Borel subgroup of $G$ containing $T$.
We have therefore associated to $\Psi$ a triple $(G,T,\varphi)$ of the required
type.

Suppose we vary the arbitrary choices made in the above construction of $(G,T,\varphi)$.
That is, suppose we choose
\begin{itemize}
\item another root datum $\Psi'$ that is $\Gal(k)\times\Gamma$-isomorphic to $\Psi$, a based root datum
$\dot\Psi'$ with underlying datum $\Psi'$ and base $\Delta'$, and hence a cocycle
$c'$ in $Z^1(k,W(\dot\Psi')^\Gamma)$;
\item another triple of $k$-groups $(G', B'_0,T'_0)$ $k$-isomorphic to $(G,B_0,T_0)$ and an identification
of $\dot\Psi(G',B'_0,T'_0)$ with $\dot\Psi'$; and
\item a $\Gal(k)$-stable pinning $\{ X_\alpha'\}$
of $G'$ relative to $B'_0$ and $T'_0$, along with the associated map
$\map{\psi'}{\Aut(\dot\Psi(G',B'_0,T'_0))}{\Aut(G',B'_0,T'_0,\{ X_\alpha'\})}$
from Remark~\ref{rem:splitting}.
\end{itemize}
We will show that these choices lead to a triple $(G',T',\varphi')$ that is equivalent to $(G,T,\varphi)$.

By assumption, there exists a $\Gal(k)\times\Gamma$-isomorphism
$\map{\lambda}{\Psi}{\Psi'}$. 
The unique element $w\in W(\Psi')$ mapping $\lambda(\Delta)$ to $\Delta'$ 
must lie in  $W(\Psi')^\Gamma$ since $\lambda(\Delta)$ and $\Delta'$ are both $\Gamma$-stable.
The map $\abmap{\lambda(\dot\Psi)}{\dot\Psi'}$ induced by $w$ is equivariant
with respect to the actions of $\Gal(k)$ on these based root data.
It follows that the composition $\dot\lambda = w\lambda$ is a $\Gal(k)\times\Gamma$-isomorphism
$\abmap{\dot\Psi}{\dot\Psi'}$. Via conjugation, $\dot\lambda$ induces 
$\Gamma$-equivariant isomorphisms
$\abmap{W(\dot\Psi)}{W(\dot\Psi')}$
and
$\map{\iotap}{Z^1(k,W(\dot\Psi))}{Z^1(k,W(\dot\Psi'))}$.

The map $\iotap(c) \colon \sigma\mapsto \dot\lambda c(\sigma)\dot\lambda\inv$
is a cocycle in $Z^1(k,W(\dot\Psi')^\Gamma)$,
cohomologous to $c'$; more precisely, one can check that for $\sigma\in\Gal(k)$, 
\begin{equation}
\label{eq:c-prime-vs-tau-c}
c'(\sigma)
= w\inv(\dot\lambda c(\sigma)\dot\lambda\inv)\sigma^{\star\prime}(w)
= w\inv(  \iotap(c)(\sigma) )\sigma^{\star\prime}(w),
\end{equation}
where $\sigma^{\star\prime}$ denotes the result of the action of
$\sigma$ on $w$, viewed as an element of $W(\dot\Psi ')$.

Use a particular $k$-isomorphism between $(G', B'_0,T'_0)$ and $(G,B_0,T_0)$ to identify these triples.
Following the above construction, we obtain a homomorphism
$\map{\varphi'}{\Gamma}{\Aut(G,B_0,T_0,\{ X'_\alpha\})}$, as well as
an element
$g'\in\bigl(G^{\varphi'(\Gamma)}\bigr)\conn$ and a $k$-torus
$T' = \lsup{g'} T_0$, analogous to $g$ and $T$, respectively.

We have a
unique element $\kappa$ of $\Aut_k(\dot\Psi(G,B_0,T_0))$ that produces a commutative square
of $\Gal(k)$-equivariant maps
\begin{equation}
\label{eq:identification}
\begin{xy}
\xymatrix{ 
\dot\Psi \ar[d]^{\dot\lambda} \ar[r] & \dot\Psi(G,B_0,T_0) \ar[d]^\kappa \\
\dot\Psi' \ar[r] & \dot\Psi(G,B_0,T_0)
} 
\end{xy}
\end{equation}
Here the horizontal arrows are the identifications chosen in the 
respective constuctions of $\varphi$ and $\varphi'$.
We therefore obtain a diagram
\begin{equation}
\label{eq:gamma-actions}
\begin{xy}
\xymatrix{ 
& \Aut(\dot\Psi) \ar[dd] \ar[r] & \Aut(\dot\Psi(G,B_0,T_0)) \ar[dd] \\
\Gamma \ar[ru] \ar[rd] & \\
& \Aut(\dot\Psi') \ar[r] & \Aut(\dot\Psi(G,B_0,T_0))
} 
\end{xy}
\end{equation}
in which the square on the right is induced by \eqref{eq:identification}
(and hence commutes),
the vertical maps are given respectively by conjugation by $\dot\lambda$ and $\kappa$,
the maps out of $\Gamma$ are given by the actions of $\Gamma$ on $\dot\Psi$ and $\dot\Psi'$,
and the triangle commutes by the $\Gamma$-equivariance of $\dot\lambda$.

Identifying $c$ and $c'$ respectively with cocycles in $Z^1(k,W(G,T_0)^{\varphi(\Gamma)})$
and $Z^1(k,W(G,T_0)^{\varphi'(\Gamma)})$ as in the above construction,
it follows from \eqref{eq:gamma-actions}
and
\eqref{eq:c-prime-vs-tau-c}
that
\begin{equation}
\label{eq:c-vs-c-prime}
c'(\sigma) = w\inv(\kappa\circ c(\sigma)\circ \kappa\inv)\sigma(w),
\end{equation}
where $\sigma(w)$ here denotes the result of $\sigma\in\Gal(k)$ acting
on the element $w \in W(\dot\Psi')$ via the identification
of this group with the concrete Weyl group $W(G,T_0)$
in \eqref{eq:gamma-actions}.

Let $n\in N_{\downstairs G}(\downstairs T_0)(k\sep)$ be a representative for $w$
and set $\mu = \psi(\kappa)\in\Aut_k(G,B_0,T_0)$.
Then by
\eqref{eq:c-vs-c-prime},
${g'}\inv\sigma(g')$ and $n\inv \mu(g\inv\sigma(g))\sigma(n)$
have the same image in $W(G,T_0)$.
Rearranging terms and letting $h = g'n\inv \mu(g)\inv$, we obtain that
$\sigma(h)$ and $h$ have the same image modulo
\begin{equation}
\label{eq:tori}
\lsup{\sigma(\mu(g)n)}T_0 = \sigma (\lsup{\mu(g)n}T_0) = \sigma (\lsup{\mu(g)}T_0)
 = \sigma (\mu (\lsup{g}T_0)) = \sigma (\mu (T)) = \mu(T).
 \end{equation}
Let $\nu = \Int(h) \circ \mu$.
Since
\begin{multline*}
\nu(T)
= \Int(h)(\mu(T))
= \Int (g'n\inv \mu(g)\inv)(\mu(T))
= \Int (g'n\inv) (\mu(\lsup{g\inv}T)) \\
= \Int (g'n\inv) (\mu(T_0))
= \Int (g'n\inv) (T_0)
= \lsup{g'}T_0
= T',
\end{multline*}
it follows from \eqref{eq:tori} that
$\nu$ gives a $k$-isomorphism $T\longrightarrow T'$.

To show that $(G',T',\varphi')$ is equivalent to $(G,T,\varphi)$, it remains to show that $\nu$ is
$\Gamma$-equivariant.
It follows from the construction of $\varphi$ that $\pi\circ\varphi$
is equal to the composition
$\Gamma\longrightarrow\Aut(\dot\Psi)\longrightarrow\Aut(\dot\Psi(G,B_0,T_0))$
appearing in \eqref{eq:gamma-actions}.
Similarly, $\pi\circ\varphi'$ is equal to the analogous composition
$\Gamma\longrightarrow\Aut(\dot\Psi')\longrightarrow\Aut(\dot\Psi(G,B_0,T_0))$. Thus for any
$\gamma\in\Gamma$, 
\[
\pi(\varphi'(\gamma)) = \kappa\circ\pi(\varphi(\gamma))\circ\kappa\inv .
\]
Applying $\psi$ to this equality and noting that  $\psi\circ\pi\circ\varphi = \varphi$ by construction,
we obtain
\begin{equation*}
\psi(\pi(\varphi'(\gamma))) = \mu\circ\varphi(\gamma)\circ\mu\inv.
\end{equation*}
Note that by definition, $\psi(f)$ and $\psi'(f)$ agree on $T_0$ for any $f\in\Aut(\dot\Psi(G,B_0,T_0))$.
Therefore, as automorphisms of $T_0$, we have
\begin{equation*}
\varphi'(\gamma) =  \psi'(\pi(\varphi'(\gamma))) = \psi(\pi(\varphi'(\gamma))) = \mu\circ\varphi(\gamma)\circ\mu\inv.
\end{equation*}
It follows that, as maps on $T$,
\begin{align*}
\lefteqn{\varphi'(\gamma)\circ\nu}\\
&= \varphi'(\gamma)\circ \Int(h)\circ\mu\\
&= \varphi'(\gamma)\circ \Int(g'n\inv \mu(g)\inv)\circ\mu\\
&= \Int(g')\circ \varphi'(\gamma)\circ \Int(\mu(g)n)\inv\circ\mu\\
&= \Int(g')\circ \mu\circ\varphi (\gamma)\circ\mu\inv\circ \Int(\mu(g)n)\inv\circ\mu \\
&= \Int(g')\circ \mu\circ\varphi (\gamma)\circ \Int(g\mu\inv(n))\inv \\
&= \Int(g')\circ \mu\circ \Int(g\mu\inv(n))\inv\circ\varphi (\gamma) ,
\end{align*}
where the last equality above comes from the fact that $g\in\downstairs G(k\sep)$ and
$\Int(\mu\inv(n))\in W(G,T_0)^{\varphi(\Gamma)}$.
Thus $\varphi'(\gamma)\circ\nu$
is equal to
\[
\Int(g'n\inv\mu(g)\inv)\circ \mu\circ\varphi (\gamma) = \nu\circ\varphi (\gamma),
\]
showing that $\nu$ is $\Gamma$-equivariant.
Therefore, $(G',T',\varphi')$ is equivalent to $(G,T,\varphi)$, and our construction induces a well-defined
map $\map{\sect_\Gamma}{\classR}{\classT}$.

We now show that $\proj_\Gamma\circ\sect_\Gamma$ is the identity map on $\classR$. Let $\Psi$
be a root datum representing some class in $\classR$, and let $(G,T,\varphi)$ be a triple representing the
image of the class of $\Psi$ under $\sect_\Gamma$. We need to show that $\Psi(G,T)$ is
$\Gal(k)\times\Gamma$-isomorphic to $\Psi$. We will make free use of the notation developed in the construction
of $\sect_\Gamma$.

The $\Gal(k)$-equivariant isomorphism of based root data $\dot\Psi\longrightarrow\dot\Psi(G,B_0,T_0)$ chosen
in the definition of $\sect_\Gamma$ is $\Gamma$-equivariant by construction (where
the action of $\Gamma$ on $\dot\Psi(G,B_0,T_0)$ is induced by $\varphi$).
We may therefore identify $\dot\Psi$ and $\dot\Psi(G,B_0,T_0)$ as based root data with
$\Gal(k)\times\Gamma$-action via this isomorphism. This allows us to identify $\Psi$ and $\Psi(G,T_0)$
as root data with $\Gamma$-action (but not necessarily with $\Gal(k)$-action since the actions
of $\Gal(k)$ on $\dot\Psi$ and $\Psi$ differ in general).

Recall the element $g\in \downstairs G(k\sep)$ chosen in the definition of $\sect_\Gamma$.
The map $\map{\Int(g)^*}{\Psi = \Psi(G,T_0)}{\Psi(G,T)}$ is $\Gamma$-equivariant since
$g$ is $\varphi(\Gamma)$-fixed. Furthermore,
$\Int(g)^*$ is $\Gal(k)$-equivariant since for $\sigma\in\Gal(k)$ and $\chi\in X^*(T_0)$,
\begin{align*}
\Int(g)^*(\sigma\chi) &= \Int(g)^*\bigl(c(\sigma)(\sigma^\star \chi)\bigr)\\
 &= \lsup{gg\inv\sigma(g)}(\sigma^\star\chi)\\
 &= \lsup{\sigma(g)}(\sigma^\star\chi)\\
 &= \sigma(\lsup g\chi)\\
 &= \sigma (\Int(g)^*(\chi)) .
\end{align*}
Thus $\Psi(G,T)$ is $\Gal(k)\times\Gamma$-isomorphic to $\Psi$, as desired.

Finally, we show that $\sect_\Gamma\circ\proj_\Gamma$ is the identity map on $\classT$.
Let $(G,T,\varphi)$ represent a class in $\classT$, and let $(G',T',\varphi')$ represent the
image of this class under $\sect_\Gamma\circ\proj_\Gamma$. Since
$\proj_\Gamma\circ (\sect_\Gamma \circ\proj_\Gamma) =
(\proj_\Gamma\circ \sect_\Gamma )\circ\proj_\Gamma = \proj_\Gamma$,
it follows that there is a $\Gal(k)\times\Gamma$ isomorphism
$\abmap{\Psi(G,T)}{\Psi(G',T')}$.
By~\cite{springer:corvallis}*{Theorem 2.9}, this isomorphism is induced
by an isomorphism $\map{\nu}{G}{G'}$ that
restricts to a $\Gamma$-equivariant $k$-isomorphism $\abmap{T}{T'}$.
Thus $(G,T,\varphi)$ and $(G',T',\varphi')$ are equivalent.
\end{proof}

\begin{rem}
\label{rem:fixed-pinning}
Observe that in the definition of the map $\sect_\Gamma$ above,
the triple $(G,T,\varphi)$ is constructed in such a way that $G$
is $k$-quasisplit and $\varphi$ fixes a $\Gal(k)$-invariant pinning of $G$.
Thus, since $\sect_\Gamma \circ \proj_\Gamma$ is the identity map on $\classT$,
we see that every equivalence class in $\classT$ contains such a triple.

Moreover, suppose that $(G, T, \varphi)$ is a triple of this kind.
Applying the construction of $\sect_\Gamma \circ \proj_\Gamma$ to this triple,
we see that the triple we obtain
is precisely $(G,T,\varphi)$,
provided that we make appropriate choices.
\end{rem}

\begin{rem}
Recall that in the proof, it is shown that if $(G,T,\varphi)$ and $(G',T',\varphi')$ are two triples
that arise by applying the $\sect_\Gamma$ construction to a root datum $\Psi$, then
$(G,T,\varphi)$ and $(G',T',\varphi')$ are equivalent. We note that the equivalence $\nu$ constructed in
this case is of a special kind. Namely, $\nu$ is of the form $\Int(h) \circ \mu$, where $h\in G'(k\sep)$ and
$\mu$ is a $k$-isomorphism from $G$ to $G'$.

Now suppose that $(G,T,\varphi)$ and $(G',T',\varphi')$ are arbitrary equivalent triples 
with the properties that $G$ and $G'$ are $k$-quasisplit and 
$\varphi$ and $\varphi'$ fix $\Gal(k)$-invariant pinnings for $G$ and $G'$, respectively.
Then combining the first part of this remark with Remark~\ref{rem:fixed-pinning}, it follows that
there is an equivalence $\nu$ between $(G,T,\varphi)$ and $(G',T',\varphi')$ of the above
special form.
\end{rem}

\begin{rem}
Suppose that $G'$ is $k$-quasisplit and $T'$ is a maximal $k$-torus of $G'$. Suppose that the finite group $\Gamma$
acts via $\Gal(k)$-equivariant automorphisms on $\Psi(G',T')$ preserving a base. Then the equivalence class of
$\Psi(G',T')$ lies in $\classR$. Applying the construction
in the definition of $\sect_\Gamma$ to $\Psi(G',T')$, we obtain a triple $(G,T,\varphi)$ where $G$ is $k$-quasisplit.
Since $\Psi(G',T')$ and $\Psi(G,T)$ are $\Gal(k)$-isomorphic, $G'$ can be taken to equal $G$. 
Moreover, if $g\in G(k\sep)$ is chosen such that $T' = \lsup g T_0$, then 
the cocycle $c$ used to define $T$
can be taken to be the image of $\sigma\mapsto g\inv\sigma(g)$ in 
$Z^1(k,W(G, T_0)^\Gamma)$.
In particular, it follows from~\cite{adler-lansky:lifting}*{Proposition 6.1} that
$T'$ is stably conjugate to $T$ and that the
$\Gal(k)\times \Gamma$-equivariant 
isomorphism between $\Psi(G,T')$ and
$\Psi(G,T)$ can be given by an inner automorphism of $G$.
(Of course, one could also construct another such triple involving a torus not necessarily
stably conjugate to $T'$ by taking the image of $(G,T,\varphi)$ under a rational automorphism of $G$.)
\end{rem}

We now turn our attention to Theorem~\ref{thm:compatibility}. 
Let $G$ be a quasisplit connected reductive $k$-group,
$B$ a Borel subgroup of $G$, and $T\subseteq B$ a maximal $k$-torus of $G$.
Let $\{ X_\alpha\}$ be a $\Gal(k)$-stable pinning of $G$ with respect
to $(B,T)$. Suppose a group $\Gamma$ acts on $G$ via $k$-automorphisms,
preserving $B$, $T$, and $\{X_\alpha\}$.
Then $\Gamma$ acts on the based root datum $\dot\Psi(G,B,T)$,
and we will freely use the notation of \S\ref{sec:root-data} in the following.

In particular, we let $\downstairs \Psi = (\downstairs X^*,\downstairs\Phi,\downstairs X_*,\downstairs\Phi^\vee)$
be the restricted root datum associated to the action of $\Gamma$ on $\Psi(G,B,T)$
by Theorem \ref{thm:restriction}.
Construct a new root datum $\downstairs \Psi'$ by replacing the root system
$\downstairs\Phi$ of $\downstairs\Psi$ by a maximal reduced subsystem
$\downstairs\Phi'$
as in Remark \ref{rem:reduced-subsystem},
and do likewise with the coroot system.

By~\cite{adler-lansky:lifting}*{Proposition~3.5},
$\downstairs G := (G^\Gamma)\conn$ is a reductive $k$-group, and 
$\downstairs T := (T^\Gamma)\conn$ is a maximal $k$-torus of $\downstairs G$. 
We may identify $\downstairs X^*$ with $X^*(\downstairs T)$.
Under this identification, the restriction $\beta_\textrm{res}$ of a root $\beta\in\Phi(G,T)$ to
$\downstairs T$ corresponds to $i^*\beta\in\downstairs X^*$.

\begin{lem}
\label{lem:weyl-fixed-pinning}
Using the above notation,
and under the above identification of $\downstairs X^*$ with $X^*(\downstairs T)$,
we have $\Psi(\downstairs G, \downstairs T) = \downstairs\Psi'$,
and $W(\downstairs G, \downstairs T) = W(G,T)^\Gamma$.
\end{lem}

\begin{proof}
Since the Weyl group of $\Psi(\downstairs G, \downstairs T)$ is
$W(\downstairs G, \downstairs T)$,
and $W(\downstairs \Psi) = W(\downstairs\Psi')$,
and
Corollary \ref{cor:weyl-isomorphism}
implies that $W(\downstairs\Psi) = W(G,T)^\Gamma$,
the claim about the Weyl groups will follow from the claim about root data.

To prove this, we reduce to the well-known case where $\Gamma$ is cyclic.
It is clear from the constructions that to show that
$\downstairs \Psi' = \Psi(\downstairs G, \downstairs T)$,
it suffices to show $\downstairs \Phi' = \Phi(\downstairs G, \downstairs T)$.
This statement follows for $G$ if it
holds for a central quotient of $G$.
Therefore, we may assume that (over $k\sep$) $G$ is a direct product of
simple groups.
We can also reduce to the case where
$\Gamma$ acts transitively on the factors of $G$.
As in the proof of \cite{adler-lansky:lifting}*{Proposition~3.5},
we may identify the factors
of $G$ with each other, and replace $\Gamma$ by a group $S\times \Gamma_1$,
where $S$ acts by permuting the factors in our product decomposition of $G$, and
$\Gamma_1$ preserves each factor and acts in the same way on each,
such that
\begin{itemize}
\item
the action of $S\times \Gamma_1$ preserves $\{X_\alpha\}$,
\item
$\downstairs G = (G^{S\times \Gamma_1})\conn$,
\item
$\downstairs T = (T^{S\times \Gamma_1})\conn$, and
\item
$\downstairs\Phi$, hence $\downstairs\Phi'$,
does not change when we replace the action of $\Gamma$ by that of $S\times\Gamma_1$.
\end{itemize}

Working in stages, we may assume that $\Gamma$ is simple.
Thus, either $\Gamma$ acts by permutation of the factors of $G$,
or $G$ has a single factor and thus is simple.
In the former case, our result is trivial, so assume that $G$ is simple.
Then $G$ has a connected Dynkin diagram, whose automorphism group is solvable.
Since $\Gamma$ embeds in this automorphism group, $\Gamma$ must be solvable hence cyclic.
It is stated without proof in~\cite{kottwitz-shelstad:twisted-endoscopy}*{\S1.1}
that in the cyclic case, $\Phi(\downstairs G, \downstairs T) = \downstairs\Phi'$
(and $W(\downstairs G, \downstairs T) = W(G,T)^\Gamma$).
We include a proof below.

It follows from~\cite{steinberg:endomorphisms}*{\S8.2(2$''''$)}
that since $\Gamma$ fixes a pinning (i.e., for each simple root $\beta \in \Phi(G,T)$,
we have that $c_\beta = 1$ in the terminology~\loccit),
for each $\beta\in\Phi(G,T)$ such that $i^*\beta\in \downstairs\Phi'$,
$\beta_{\textrm{res}}$ belongs to $\Phi(\downstairs G, \downstairs T)$.
Since, by definition, $\downstairs\Phi' \subseteq i^*(\Phi(G,T))$,
and since $\beta_{\textrm{res}}$ corresponds to $i^*\beta$ under our identification
of $X^*(\downstairs T)$ with $\downstairs X^*$, this
shows that $\downstairs\Phi'\subseteq\Phi(\downstairs G,\downstairs T)$.
On the other hand, by~\cite{adler-lansky:lifting}*{Proposition 3.5(iv)}, every
root in $\Phi(\downstairs G,\downstairs T)$ is the restriction of a root in
$\Phi(G,T)$, so $\Phi(\downstairs G,\downstairs T)\subseteq i^*(\Phi(G,T)) = \downstairs\Phi$.
But $\downstairs\Phi'$ is a maximal reduced subsystem of $\downstairs\Phi$, so
$\Phi(\downstairs G,\downstairs T) = \downstairs\Phi'$, which concludes the proof.
\end{proof}

\begin{proof}[Proof of Theorem \ref{thm:compatibility}]
Consider a class in $\classT$.
From Remark \ref{rem:fixed-pinning},
we can represent this class by a triple $(G,T,\varphi)$, where
$G$ is $k$-quasisplit and the action $\varphi$ of $\Gamma$ on $G$
fixes a $\Gal(k)$-invariant pinning.
Let $\downstairs G = (G^{\varphi(\Gamma)})\conn$
and $\downstairs T = (T^{\varphi(\Gamma)})\conn$.
In order to prove our result, we must show that
the reduced root datum $\downstairs\Psi'$ contained in the restricted
root datum $\downstairs\Psi$ of Theorem \ref{thm:restriction}
is equivalent to $\Psi(\downstairs G, \downstairs T)$.
This is the content of
Lemma \ref{lem:weyl-fixed-pinning}.
\end{proof}

\section{Remarks on Cohomological Parametrizations}
\label{sec:cohomology}
In this section, we describe partitions of $\classR$ and $\classT$ into blocks, each of which
can be parametrized
cohomologically.

Let $\dot\Psi_0$ be a based root datum with $\Gal(k)\times \Gamma$ action, i.e., the
distinguished base is $\Gal(k)\times \Gamma$-stable.
Let $\Psi_0$ be the root datum underlying $\dot\Psi_0$.
Let $\Psi$
be a root datum with $\Gal(k)\times \Gamma$-action such that
\begin{itemize}
\item $\Gamma$ preserves a base of $\Psi$; and
\item
with respect to some (hence any) choice of $\Gamma$-invariant base of $\Psi$,
the associated based root datum $\dot\Psi$ (with action $\sigma\mapsto\sigma^\star$ as in
\eqref{eq:*-action})
is $\Gal(k)\times\Gamma$-isomorphic to
$\dot\Psi_0$.
\end{itemize}
The collection of such root data $\Psi$ is closed under
$\Gal(k)\times\Gamma$-isomorphism 
(the relation of equivalence from~\S\ref{sec:intro}),
and we let $\mathscr{R}_\Gamma^{\dot\Psi_0}$ denote
the set of equivalence classes of data with the above properties.

Let $\nu$ be a $\Gal(k)\times\Gamma$-isomorphism $\abmap{\dot\Psi_0}{\dot\Psi}$
implementing the above equivalence.
For $\sigma\in\Gal(k)$, let $\sigma(\nu)$ be the map $\sigma\circ\nu\circ\sigma\inv$,
viewed as an isomorphism $\abmap{\Psi_0}{\Psi}$.
In particular, the action of $\sigma$ on $\Psi$ here is via the original action of $\Gal(k)$,
not the action $\sigma\mapsto\sigma^\star$ on $\dot\Psi$.
Evidently, the map $\nu\inv\sigma(\nu)$ lies in the group $\Aut_\Gamma(\Psi_0)$ of
$\Gamma$-equivariant automorphisms of $\Psi_0$. 
It is not difficult to show that in fact, $\nu\inv\sigma(\nu)$ lies
in the subgroup $W(\Psi_0)^\Gamma$ of $\Aut_\Gamma(\Psi_0)$ and,
moreover, that $\sigma\mapsto\nu\inv\sigma(\nu)$ determines
a cocycle $c$ in $Z^1(k,W(\Psi_0)^\Gamma)$.

Altering the choice of $\Psi$ in its class or the choice of $\nu$ has the effect of
replacing $c$ by $\sigma\mapsto\kappa\inv \circ c(\sigma) \circ\sigma(\kappa)$ for some
$\kappa\in\Aut_\Gamma(\Psi_0)$. Hence the class of $\Psi$ determines an element of 
$$\classH^{\dot\Psi_0}:=\Image\bigl[ H^1(k,W(\Psi_0)^\Gamma)\longrightarrow H^1(k,\Aut_\Gamma(\Psi_0)\bigr].$$
It is readily seen that, conversely, an element of this image determines a root datum $\Psi$ as above,
unique up to $\Gal(k)\times\Gamma$-isomorphism. Thus we have defined a one-to-one correspondence
$\map{a_\Gamma}{\classR^{\dot\Psi_0}}{\classH^{\dot\Psi_0}}$.

We now describe a corresponding parametrization for blocks in $\classT$.
Define $\classT^{\dot\Psi_0}$ to be $s_\Gamma\big(\classR^{\dot\Psi_0}\big)$,
i.e., the set of classes in $\classT$ containing triples whose associated based
root data are $\Gal(k)\times\Gamma$-isomorphic to $\dot\Psi_0$.

Recall that the definition of $s_\Gamma$ in Theorem~\ref{thm:root-data-to-group}
involves the choice of a triple $(G,B_0,T_0)$ consisting of a $k$-quasisplit connected reductive group $G$,
a Borel $k$-subgroup $B_0$ of $G$,
and a maximal $k$-torus $T_0$ of $B_0$. We may identify
the based root data (with $\Gal(k)\times\Gamma$-actions) $\dot\Psi (G,B_0,T_0)$ and $\dot\Psi_0$.
We observe that the construction of the triple
$(G,T,\varphi)$ in the definition of $s_\Gamma$
involves passing from the given root datum $\Psi$ to a cocycle
$c$ in $Z^1(k,W(G,T_0)^{\varphi(\Gamma)}) = Z^1(k,W(\dot\Psi_0)^\Gamma)$.
It can be checked that the passage from $\Psi$ to $c$ is precisely the one
occurring in the above definition of the map $a_\Gamma$.

The proof of Theorem~\ref{thm:root-data-to-group}
proceeds to define $(G,T,\varphi)$ in such a way that it depends only on $c$, and not
directly on $\Psi$. Moreover, it follows from the proof that
the equivalence class of $(G,T,\varphi)$ depends only on the class of $c$
in $H^1(k,\Aut_\Gamma(\Psi_0))$. Therefore, we obtain a map
$\map{b_\Gamma}{\classH^{\dot\Psi_0}}{\classT^{\dot\Psi_0}}$ such
that $s_\Gamma$ factors through $\classH^{\dot\Psi_0}$ according to the
commutative diagram
$$
\begin{xy}
\xymatrix{
\classR^{\dot\Psi_0}
\ar@/^.3pc/[rr]^{\sect_\Gamma}
\ar[dr]^{a_\Gamma}
&&
\classT^{\dot\Psi_0}
\ar@/^.3pc/[ll]^{\proj_\Gamma}
\\
&\classH^{\dot\Psi_0}
\ar@/^.3pc/[ur]^{b_\Gamma}
}
\end{xy}
$$
In particular, since $a_\Gamma$ and $s_\Gamma$ are bijections, $b_\Gamma$ must also be a bijection.

\end{document}